\documentclass[12pt]{amsart}
\usepackage{amsmath}
\usepackage{amsfonts, amssymb, euscript, mathrsfs}
\usepackage{graphicx}
\usepackage[usenames, dvipsnames]{color}

\usepackage{tikz}
\usepackage{enumerate}
\usepackage[tmargin=1in,bmargin=1in,lmargin=0.9in,rmargin=0.9in]{geometry}
\usepackage{aliascnt}
\usepackage{hyperref}
\usepackage{verbatim}

\allowdisplaybreaks[4]

\newcommand{\BR}{\mathbb{R}}

\newcommand{\SL}{\sum\limits}

\newcommand{\al}{\alpha}

\newcommand{\CF}{\mathcal F}

\newcommand{\MP}{\mathbf P}

\newcommand{\Oa}{\Omega}

\newcommand{\si}{\sigma}

\renewcommand{\phi}{\varphi}

\renewcommand{\comment}[1]{}

\newcommand{\mP}{\mathbf{p}}

\newcommand{\md}{\mathrm{d}}

\begin{document}


\theoremstyle{plain}
\newtheorem{thm}{Theorem}[section]
\newtheorem*{thmnonumber}{Theorem}
\newtheorem{lemma}[thm]{Lemma}
\newtheorem{prop}[thm]{Proposition}
\newtheorem{cor}[thm]{Corollary}
\newtheorem{open}[thm]{Open Problem}

\theoremstyle{definition}
\newtheorem{defn}{Definition}
\newtheorem{asmp}{Assumption}
\newtheorem{notn}{Notation}
\newtheorem{prb}{Problem}

\theoremstyle{remark}
\newtheorem{rmk}{Remark}
\newtheorem{exm}{Example}
\newtheorem{clm}{Claim}

\title[Absence of collisions for competing Brownian particles]{Yet Another Condition for Absence of Collisions\\ for Competing Brownian Particles}

\author{Tomoyuki Ichiba}

\address{University of California, Santa Barbara, Department of Statistics and Applied Probability}

\email{ichiba@pstat.ucsb.edu}

\author{Andrey Sarantsev}

\address{University of California, Santa Barbara, Department of Statistics and Applied Probability}

\email{sarantsev@pstat.ucsb.edu}

\date{Version 14. January 11, 2017}



\begin{abstract}
Consider a finite system of rank-based competing Brownian particles, where the drift and diffusion of each particle depend only on its current rank relative to other particles. We present a simple sufficient condition for absence of multiple collisions of a given order, continuing the earlier work by Bruggeman and Sarantsev (2015). Unlike in that paper, this new condition works even for infinite systems. 
\end{abstract}

\maketitle

\thispagestyle{empty}

\section{Introduction and the Main Result}

Consider a system of $N$ Brownian particles $X_1(t), \ldots, X_N(t)$, $t\ge 0$ on the real line.  Each particle $X_i(t)$ of name $\,i\,$ evolves as a Brownian motion with rank-dependent drift coefficient $g_k$ and diffusion coefficient $\sigma_k^2$, where $k$ is the current {\it relative rank} of $X_i$ at time $t$. Namely, at each moment we {\it rank} particles from bottom to top:
$$
Y_{1}(t) \, :=\, X_{(1)}(t) \le \ldots \le X_{(N)}(t) \, =:\, Y_{N} (t),
$$
so that the particle occupying the lowest position has rank $1$, the next particle has rank $2$, etc. If two or more particles are {\it tied}, that is, they occupy the same position at the same time, then we {\it resolve ties in lexicographic order}\, :  particles $X_i$ with lower {\it names} 
are assigned lower {\it ranks}.

We can make this description more formal, giving an SDE which governs this system $X(t) = (X_1(t), \ldots, X_N(t))$. 
For any vector $x = (x_1, \ldots, x_N) \in \BR^N$, there exists a unique {\it ranking permutation}: a one-to-one mapping $\mP_x : \{1, \ldots, N\} \to \{1, \ldots, N\}$, with the following properties:

(a) $x_{\mP_x(i)} \le x_{\mP_x(j)}$ for $1 \le i < j \le N$;

(b) if $x_{\mP_x(i)} = x_{\mP_x(j)}$ for $1 \le i < j \le N$, then $\mP_x(i) < \mP_x(j)$. 

\begin{defn} On a filtered probability space $(\Oa, \CF, (\CF_t)_{t \ge 0}, \MP)$, with the filtration satisfying the usual conditions, we consider a continuous, adapted $\BR^N$-valued process $X = (X(t), t \ge 0)$, $X(t) = (X_1(t), \ldots, X_N(t))$, which satisfies the following SDE: for $i = 1, \ldots, N$, 
\begin{equation} 
\md X_i(t) = \sum\limits_{k=1}^N1\left(\mP_t(k) = i\right)\left(g_k\md t + \si_k\md W_i(t)\right) , 
\label{eq:SDE-basic}
\end{equation}
where $W = (W_1, \ldots, W_N)$ is an $N$-dimensional Brownian motion, $\,1(\cdot)\,$ is the indicator function, and $\mP_t := \mP_{X(t)}$. Then the process $X$ is called a {\it system of $N$ competing Brownian particles} with {\it drift coefficients} $g_1, \ldots, g_N \in \mathbb R$, and {\it diffusion coefficients} $\si_1^2, \ldots, \si_N^2 > 0$, and $Y (t) \, :=\, (Y_{1}(t), \ldots , Y_{N}(t)) $ with $Y_{k}(t) := X_{\mP_{t}(k)}(t) $, $k=1, \ldots , N$, $t \ge 0$ is called a system of $N$ {\it ranked particles}. 
\label{defn:CBP}
\end{defn}

For any values $g_1, \ldots, g_N \in \BR$, $\si_1^{2}, \ldots, \si_N^{2} > 0$, there exists a (unique in law) weak solution to this SDE (\ref{eq:SDE-basic}); this result follows from \cite{Bass1987}. 



These systems were introduced in \cite{BFK2005, FernholzBook} and further studied in \cite{5people,  IPS2012, JM2008, PP2008,  MyOwn2, MyOwn10}. In particular, they are used as a model for financial markets with $N$ stocks, where $Y_i(t) = \exp\left(X_i(t)\right)$ is the capitalization of the $i$th stock. One observed real-world market feature is that stocks with smaller capitalizations have larger volatilities; we can model this by taking $\sigma_1 > \sigma_2 > \ldots > \sigma_N$. Another feature is as follows: consider the market weights $\mu_k(t) = Y_k(t)/(Y_1(t) + \ldots + Y_N(t))$ and rank them from top to bottom; then plot these ranked market weights versus their ranks on the log-log plot. The resulting graph will be close to a straight line. This can also be described by the market model of competing Brownian particles, see \cite{CP2010}. Other applications of competing Brownian particles for market models include \cite{JR2013b, MyOwn4}. Let us also mention another application of competing Brownian particles: this is a discrete analogue of the \textsc{McKean-Vlasov} equation, which governs a nonlinear diffusion process, and it can approximate this process, see \cite{4people, JR2013a, S2012}. 

Of particular interest in such system of competing Brownian particles are {\it triple and multiple collisions}. A {\it triple collision} occurs when three particles occupy the same position at the same moment. A {\it simultaneous collision} occurs when two particles collide and at the same moment other two particles collide. A triple collision is a particular case of a simultaneous collision. 



It was shown in \cite{IKS2013} that a strong solution to the SDE governing the system of competing Brownian particles exists up to the first moment of a triple collision. Whether there is a strong solution after this triple collision is still, to our best knowledge, an open question. 

The question whether triple collisions occur in this system of competing Brownian particles 
has attracted considerable attention in recent years. Progressively better results were established in the papers \cite{IK2010, IKS2013} until, in the paper \cite{MyOwn3}, a necessary and sufficient condition for a.s. absence of triple collisions was found: the sequence $(\si_1^2, \ldots, \si^2_N)$ must be {\it concave}, that is, 
$$
\si_k^2 \ge \frac12\left(\si_{k-1}^2 + \si_{k-1}^2\right),\ \ k = 2, \ldots, N-1.
$$
In addition, if there are a.s. no triple collisions, then there are a.s. no simultaneous collisions. 

It is much harder, however, to study {\it multiple collisions}, when four, five or more particles collide. Let us state a formal definition.

\begin{defn} \label{Def1} We say there are {\it no collisions of particles with ranks} $k_-, \ldots, k_+ \in \{1, \ldots, N\}$, if 
$$
\MP\left(\exists\, t > 0:\ Y_{k-}(t) = \ldots = Y_{k+}(t)\right) = 0.
$$
There are {\it no $n$-tuple collisions} if there are no collisions of particles with ranks $k_- <  \ldots < k_+$ for every pair of $k_-, k_+ = 1, \ldots, N$ such that $k_+ - k_- = n - 1$; or, equivalently, if for every subset $I = \{i_1, \ldots, i_n\} \subseteq \{1, \ldots, N\}$ of $n$ elements we have: 
$$
\MP\left(\exists\, t > 0:\ X_{i_1}(t) = \ldots = X_{i_n}(t)\right) = 0.
$$
If there are no $n$-tuple collisions with $n = N$, we say that there are {\it no total collisions}. 
\end{defn}

These topics were studied in \cite{MyOwn5}. It was shown there that these properties are independent of the initial condition, as well as of the drift coefficients $g_1, \ldots, g_N$. However, only sufficient conditions are known for lack of quadruple, quintuple and other collisions. They are rather complicated: for example, lack of quadruple collisions for a system of $N = 5$ particles involves $17$ inequalities on the diffusion coefficients $\si_1^2, \ldots, \si_N^2$. In this short paper, we attempt to provide some easy-to-verify sufficient conditions. One of them is stated in Theorem~\ref{thm:main}, and proved in Section 2. This paper extends from \cite{MyOwn5} and complements its results. In Remark~\ref{rmk:comp} later, we compare our new results with those in  \cite{MyOwn5} for the simplest case: total collisions for $N \ge 4$ competing Brownian particles, and we see that the new results do not follow from the previous ones. The paper \cite{MyOwn5} also contains sufficient conditions for absence of more complicated collisions, such as 
\begin{equation} \label{eq:multi} 
\MP \left (Y_1(t) = Y_2(t) \quad  \text{and}  \quad Y_4(t) = Y_5(t) = Y_6(t) \right) \, =\, 0. 
\end{equation}
We touch this subject in this paper in Section 3. In Section 4, we study multiple collisions for infinite systems of competing Brownian particles. We are able to do this in this paper (as opposed to the previous paper \cite{MyOwn5}), because the coefficient $(n-1)/2$ in  condition~(\eqref{eq:main} defined below) does not depend on the total number $N$ of particles.

The main result of the paper is as follows. 

\begin{thm} \label{thm1.1}For $n \ge 4$, there are no $n$-tuple collisions in Definition \ref{Def1}, if $\,\sigma_{1}^{2}, \ldots , \sigma_{n}^{2}\,$ satisfy 
\begin{equation}
\label{eq:main}
\max\limits_{1 \le k \le N}\si_k^2 < \frac{n-1}2\min\limits_{1 \le k \le N}\si_k^2.
\end{equation}
\label{thm:main}
\end{thm}

\begin{rmk} Note that this is not a necessary condition. For example, consider $n = N = 4$ and $\si_1^2 = 2$, $\si_2^2 = \si_3^2 = \si_4^2 = 1$. Then condition~\eqref{eq:main} does not hold, but conditions of \cite[Theorem 1.1]{MyOwn5} hold, and therefore, there are no $4$-tuple (in this case, total) collisions. 
\end{rmk}

\section{Proof of Theorem~\ref{thm:main}}

We present the proof of Theorem~\ref{thm:main} as a sequence of lemmata, which are of their own interest. Define
$$
\Pi_N := \{x = (x_1, \ldots, x_N)' \in \BR^N\mid x_1 + \ldots + x_N = 0,\ \ x_1^2 + \ldots + x_N^2 = 1\}.
$$

\begin{lemma} There are no total collisions in Definition \ref{Def1}, if $\,\sigma_{1}^{2}, \ldots , \sigma_{N}^{2}\,$ satisfy 
$$
\max\limits_{x \in \Pi_N}\SL_{k=1}^N\si_k^2x_k^2 < \frac{N-1}{2N}\SL_{k=1}^N\si_k^2.
$$
\label{lemma:total}
\end{lemma}

The proof of Lemma~\ref{lemma:total} was already given as part of the main result in \cite[pp.238-240]{IKS2013}. The following is an immediate corollary.

\begin{cor} \label{cor2.2} There are no total collisions in Definition \ref{Def1}, if $\,\sigma_{1}^{2}, \ldots , \sigma_{N}^{2}\,$ satisfy 

\begin{equation} \label{eq2.2}
\max\limits_{1 \le k \le N}\si_k^2 \le \frac{N-1}{2N}\SL_{k=1}^N\si_k^2.
\end{equation}
\label{cor:total}
\end{cor}

In particular, for a system of $N = 4$ competing Brownian particles, this condition (\ref{eq2.2}) takes the form
\begin{equation}
\label{eq:total-4}
\max\left(\si_1^2, \si_2^2, \si_3^2, \si_4^2\right) < \frac38\left(\si_1^2 + \si_2^2 + \si_3^2 + \si_4^2\right).
\end{equation}

\begin{rmk}
\label{rmk:comp}
This new sufficient condition does not follow from the results of \cite[Theorem 1.1, Theorem 1.3]{MyOwn5}. Take the following example:
$$
\si_1^2 = \si_4^2 = 1,\ \ \si_2^2 = \si_3^2 = \frac12.
$$
This example satisfies~\eqref{eq:total-4}, but not the assumptions of the previous results in \cite[Theorem 1.1, Theorem 1.3]{MyOwn5}. More generally, the set of $\,\si_{1}^{2}, \ldots , \si_{N}^{2}\,$ with 
\[
\si_{1}^{2} \, =\, \si_{N}^{2} \, =\, 1 \, , \ \ \si_{2}^{2} \, =\, \cdots = \si_{N-1}^{2} \, =\,  \frac{\,1\,}{\,N-2\,} \, 
\]
satisfies (\ref{eq2.2}) and hence by Corollary \ref{cor2.2} there are no total collisions. This result is not covered by the previous results in \cite{MyOwn5}. \hfill $\square$
\end{rmk}

The next corollary trivially follows from Corollary~\ref{cor:total}. 

\begin{cor}There are no total collisions in Definition \ref{Def1}, if $\,\sigma_{1}^{2}, \ldots , \sigma_{N}^{2}\,$ satisfy 
$$
\max\limits_{1 \le k \le N}\si_k^2 \le \frac{N-1}{2}\min_{1 \le k \le N} \si_k^2.
$$
\label{cor:total-easy}
\end{cor}

The next lemma establishes a link between total collisions and multiple collisions. 

\begin{lemma} There are no collisions between particles with ranks $k_- <  \cdots < k_+$, if for every pair $l_-, l_+ \in \{ 1, \ldots, N\}$ such that 
\begin{equation}
\label{eq:pair}
1 \le l_- \le k_- < k_+ \le l_+ \le N,
\end{equation}
the system of $(l_+ - l_- + 1)$ competing Brownian particles with diffusion coefficients $\si_{l_-}^2, \ldots, \si_{l_+}^2$ does not have total collisions.
\label{lemma:total2multiple}
\end{lemma}

\begin{proof} Our arguments are similar to the ones in the proof of \cite[Theorem 3.8]{MyOwn5}. Assume there is, in fact, a collision between particles with ranks $  k_- < \cdots  < k_+ $: 
\begin{equation}
\label{eq:0}
Y_{k-}(t) = \cdots = Y_{k+}(t)\ \ \mbox{for some}\ \ t > 0.
\end{equation}
Then there exist 
\begin{equation}
\label{eq:1}
l_- \in \{ 1, \ldots, k_-\} \ \ \mbox{and}\ \ l_+ \in \{ k_+, \ldots, N \} \ \ \mbox{such that}
\end{equation}
\begin{equation}
\label{eq:2}
Y_{l_- - 1}(t) < Y_{l_-}(t) = \cdots = Y_{l_+}(t) < Y_{l_++1}(t) \ \ \mbox{for some}\ \ t > 0.
\end{equation}
(For consistency of notation, we let $Y_0(t) := -\infty$ and $Y_{N+1}(t) := +\infty$.) 
From~\eqref{eq:2}, by continuity of ranked processes $Y(\cdot)$, there exists a rational number $q (\ge 0)$ such that 
\begin{equation}
\label{eq:3}
Y_{l_- - 1}(s) < Y_{l_-}(s) = \cdots = Y_{l_+}(s) < Y_{l_++1}(s),\ \ \text{ for every }  \, \, s \in [q, t].
\end{equation}
Let $J$ be the set of the names of particles with ranks $l_-, \ldots, l_+$ at the time $s \in [q, t]$. Obviously, there are $l_+ - l_- + 1$ elements in this set. Because of~\eqref{eq:3}, this set is independent of choice of $s \in [q, t]$, that is, these particles do not collide with other particles on this time interval $[q, t]$, and only then they can exchange ranks. Thus, the particles $(X_i(q + \cdot),\, i \in J)$ itself form a system of $l_+ - l_- + 1$ competing Brownian particles with diffusion coefficients $\si_{l_-}^2, \ldots, \si_{l_+}^2$, and it experienced a total collision at the moment $t - q$. This event, however, has probability zero, because of the assumption. Taking the countable union over all rational $q \ge 0$ and over all $l_-, l_+$ from~\eqref{eq:1},  we conclude the event in~\eqref{eq:0} has probability zero. This completes the proof. 
\end{proof}

Combining Lemma~\ref{lemma:total2multiple} with Corollary~\ref{cor:total-easy}, we complete the proof of Theorem~\ref{thm:main}.

\section{Other Types of Collisions}

In the paper \cite{MyOwn5}, we also consider more complicated collisions, such as the collision event 
in~\eqref{eq:multi}. The results of Theorem~\ref{thm:main} and Section 2 can be useful  here, because we solve this problem in part by reducing it to a study of total collisions. Here, we do not study general complicated collisions. Rather, let us state briefly the results for the case $N = 4$ (four particles), where they are the simplest but still interesting.

For a system of $N = 4$ competing Brownian particles, we have the following types of collisions:
\begin{equation} \label{eq:total}
Y_1(t) = Y_2(t) = Y_3(t) = Y_4(t)\ \ \mbox{(the total collision);}
\end{equation}
\begin{equation}
\label{eq:triple-123}
Y_1(t) = Y_2(t) = Y_3(t)\ \ \mbox{(a triple collision);}
\end{equation}
\begin{equation}
\label{eq:triple-234}
Y_2(t) = Y_3(t) = Y_4(t)\ \ \mbox{(a triple collision);}
\end{equation}
\begin{equation}
\label{eq:simult-12-34}
Y_1(t) = Y_2(t)\ \ \mbox{and}\ \ Y_3(t) = Y_4(t)\ \ \mbox{(the simultaneous collision).}
\end{equation}
For the total collision of type (\ref{eq:total}), a sufficient condition to avoid them is given in Theorem~\ref{thm:main}, and in the intermediate results of Section 2; for example~\eqref{eq:total-4}. For triple collisions as in~\eqref{eq:triple-123},~\eqref{eq:triple-234}, as well as the simultaneous collision in~\eqref{eq:simult-12-34}, we can find sufficient condition to avoid them, using the same techniques as in \cite{MyOwn5}. Here is a result. 

\begin{lemma}
(a) Assume that, in addition to~\eqref{eq:total-4}, we have:
$$
\si_2^2 \ge \frac12\left(\si_1^2 + \si_3^2\right).
$$
Then there are no triple collisions of the type~\eqref{eq:triple-123}.

\medskip

(b) Assume that, in addition to~\eqref{eq:total-4}, we have:
$$
\si_3^2 \ge \frac12\left(\si_2^2 + \si_4^2\right).
$$
Then there are no triple collisions of the type~\eqref{eq:triple-234}.

\medskip

(c) Under condition~\eqref{eq:total-4}, there are no simultaneous collisions of the type \eqref{eq:simult-12-34}.
\end{lemma}

The proof is similar to that of \cite[Theorem 3.8]{MyOwn5} and is omitted.

\section{Infinite Systems of Competing Brownian Particles}

Consider an infinite sequence $(X_i)_{i \ge 1}$ of continuous, adapted, real-valued processes $X_i = (X_i(t), t \ge 0)$, $i \ge 1$. Assume that for every $t \ge 0$, we can rank $X_i(t),\, i = 1, 2, \ldots$ from the bottom upward:
$$
X_{(1)}(t) \le X_{(2)}(t) \le \cdots
$$
As in the finite case, we resolve ties in lexicographic order. Let us fix parameters $g_n \in \BR$ and $\si_n > 0$, for $n = 1, 2, \ldots$ and take countably many i.i.d. Brownian motions $W_1, W_2, \ldots$ Assume the infinite sequence $X_1, X_2, \ldots$ of processes satisfy
$$
\md X_i(t) = \SL_{k=1}^{\infty}1\left(X_i\ \mbox{has rank}\ k\ \mbox{at time}\ t\right)\left(g_k\md t + \si_k\md W_i(t)\right),\ \ i = 1, 2, \ldots
$$
Then the $\BR^{\infty}$-valued process $X = (X(t), t \ge 0)$, with $X(t) = (X_i(t))_{i \ge 1}$, is called an {\it infinite system of competing Brownian particles} with {\it drift coefficients} $g_1, g_2, \ldots$ and {\it diffusion coefficients} $\si_1^2, \si_2^2, \ldots$ Similarly to the finite system, each $X_i = (X_i(t), t \ge 0)$ is called the {\it $i$th named particle}, and each $Y_k = (Y_k(t):= X_{(k)}(t), t \ge 0)$ is called the {\it $k$th ranked particle}. These systems were studied in \cite{IKS2013, MyOwn6, S2011}. Multiple collisions are defined similarly to finite systems. Triple collisions were studied in \cite[Section 5]{MyOwn6}, however, quadruple and higher-order collisions have not been yet studied for this type of infinite systems. 

The following existence and uniqueness results were proved in \cite{IKS2013, MyOwn6}.

\begin{prop} Assume the initial condition $x = (x_i)_{i \ge 1} = X(0)$ satisfies
\begin{equation}
\label{eq:series}
\SL_{i=1}^{\infty}e^{-\al x_i^2} < \infty\ \ \mbox{for all}\ \ \al > 0.
\end{equation}

\smallskip

(a) There exists a weak version of the infinite system if  
\begin{equation}
\label{eq:bounded}
\sup\limits_{n \ge 1}|g_n| < \infty\ \ \mbox{and}\ \ \sup\limits_{n \ge 1}\si_n^2 < \infty.
\end{equation}

\smallskip

(b) Under the following stronger condition, this weak version is unique in law:
$$
g_{n_0} = g_{n_0+1} = \ldots\ \ \mbox{and}\ \ \si_{n_0} = \si_{n_0+1} = \ldots\ \ \mbox{for some}\ \ n_0 \ge 1 \, .
$$
\end{prop}

Now, a remarkable fact is that the condition~\eqref{eq:main} in Theorem \ref{thm1.1} is determined independently of the number $N$ of particles. This observation allows us to prove the absence of $n$-tuple collision in Definition \ref{Def1} under the condition (\ref{AnC}) below even for infinite systems. 

\begin{thm}
\label{thm:infinite}
Take any version of an infinite system of competing Brownian particles with drift and diffusion coefficients satisfying~\eqref{eq:bounded}, with the initial configuration satisfying~\eqref{eq:series}. Take an $n \ge 4$ and assume 
\begin{equation} \label{AnC}
\sup\limits_{k \ge 1}\si_k^2 < \frac{n-1}2\inf\limits_{k \ge 1}\si_k^2.
\end{equation}
Then there are no $n$-tuple collisions. 
\end{thm}

\begin{proof} It is known from \cite[Lemma 3.4]{MyOwn6} that for every $T > 0$ and $u \in \BR$, there are a.s. only finitely many particles $X_i$ such that $\max_{0 \le t \le T}X_i(t) \le u$. Assume there is an $n$-tuple collision at time $t > 0$. Then there exist finite $l_-, l_+$ such that~\eqref{eq:2} holds. The rest of the proof follows the proof of Lemma~\ref{lemma:total2multiple}. 
\end{proof}

\section*{Acknowledgement} The authors are thankful to an anonymous reviewer for her/his careful reading and comments, in particular, a comment on a necessary condition for $n$-tuple collisions. This work was partially funded by NSF grant DMS 1409434, PI \textsc{Jean-Pierre Fouque} and by NSF grants DMS 1313373 and 1615229. 

\end{document}